\documentclass[a4paper, 11pt]{amsart}
\usepackage{amsmath, amsthm, amscd, amssymb, amsfonts, amsxtra, amssymb, latexsym}
\usepackage{enumerate}
\usepackage{verbatim}
\usepackage{inputenc, t1enc}
\usepackage{hyperref}
\hypersetup{colorlinks = true,	allcolors  = blue}

\newcommand{\Z}{\mathbb{Z}}
\newcommand{\N}{\mathbb{N}}
\newcommand{\ff}{\mathbb{F}}

\newcommand{\G}{\Gamma}

\newcommand{\sk}{\smallskip}

\newtheorem{thm}{Theorem}[section]
\newtheorem{prop}[thm]{Proposition}
\newtheorem{lem}[thm]{Lemma}
\newtheorem{coro}[thm]{Corollary}

\theoremstyle{definition}
\newtheorem{rem}[thm]{Remark}
\newtheorem{exam}[thm]{Example}
\newtheorem{defi}[thm]{Definition}

\theoremstyle{remark}

\usepackage{color}

\begin{document}
\numberwithin{equation}{section}
\title[Waring's problem and generalized Paley graphs]{The Waring's problem over finite fields \\ through generalized Paley graphs}
\author[R.A.\@ Podest\'a, D.E.\@ Videla]{Ricardo A.\@ Podest\'a, Denis E.\@ Videla}
\dedicatory{\today}
\thanks{Partially supported by CONICET, FONCyT and SECyT-UNC}
\address{Ricardo A.\@ Podest\'a, FaMAF - CIEM (CONICET), Universidad Nacional de C\'ordoba, Av.\@ Medi\-na Allende 2144, Ciudad Universitaria, (5000), C\'ordoba, Argentina. 
{\it E-mail: podesta@famaf.unc.edu.ar}}
\address{Denis E.\@ Videla,  FaMAF - CIEM (CONICET), Universidad Nacional de C\'ordoba, Av.\@ Medina Allende 2144, Ciudad Universitaria, (5000), C\'ordoba, Argentina. 
{\it E-mail: devidela@famaf.unc.edu.ar}}

\begin{abstract}
We show that the Waring number over a finite field $\ff_q$, denoted $g(k,q)$, when exists coincides with the diameter of the generalized Paley graph $\G(k,q)=Cay(\ff_{q},R_k)$ with $R_k=\{x^k : x\in \ff_q^*\}$. We find infinite new families of exact values of $g(k,q)$ from  a characterization of graphs $\G(k,q)$ which are also Hamming graphs proved by Lim and Praeger in 2009.  
Then, we show that every positive integer is the Waring number for some pair $(k,q)$ with $q$ not a prime. Finally, we find a lower bound for $g(k,p)$ with $p$ prime by using that $\G(k,p)$ is a circulant graph in this case.
\end{abstract}

\maketitle

\section{Introduction}

\subsubsection*{Motivation and historical background}
The classical problem, introduced in 1770 by Edward Waring, asks whether given a natural number $k$, there is a number $g(k)$ such that every natural number can be written as the sum of at most a number $g(k)$ of $k$-th powers. For instance, $g(1)=1$, $g(2)=4$, $g(3)=9$. 
This fact was proved by Hilbert in 1909 and it is known as the Hilbert-Waring theorem  since then.

In the context of finite fields, given a finite field $\ff_q$ and a positive integer $k$, the problem is to decide if it is possible to express every element of the field as a sum of $k$-th powers in the field. In this case, the Waring number $g(k,q)$ is the minimal value 
$s$ such that every element of $\ff_q$ is a sum of a number $s$ of $k$-th powers. 
It is not difficult to see that 
$$g(k,q)=g(\gcd(k,q-1),q),$$ 
and thus we will always assume that $k\mid q-1$ (if $(k,q-1)=1$ then $g(k,q)=g(1,q)=1$). 
Notice that the Waring number does not exist for every finite field, for instance if we take $k=4$ and $q=9$
(see Lemma~3.1).

There are three general methods to estimate $g(k,q)$, and if possible to calculate it: ($a$) additive combinatorics, 
($b$) circle methods (exponential sums) and ($c$) lattice methods. We propose to use a new strategy, 
based on the computation of the diameter of Cayley graphs whose vertex sets are the finite fields and whose connection set is the multiplicative subgroup of $k$-th powers, the so called generalized Paley graphs. 
This point of view was previously explored by Yahya O.\@ Hamidoune (according to Garc\'ia-Sol\'e) and by Christine Garc\'ia and Patrick Sol\'e (see \cite{GS}) in the case $q=p$ is a prime, and used by the last authors in \cite{GS} to find a lower bound for $g(k,p)$.

The goal of the paper is to give some new exact values for $g(k,p^m)$  with $m\in \N$ and a lower bound for $g(k,p)$, where $p$ is a prime.

\subsubsection*{Outline and results}
We now summarize the results in the paper. In Section 2, we briefly survey the results about Waring numbers over finite fields in the literature, reflecting the fact that, although there are many upper bounds for $g(k,q)$, very few exact formulas and lower bounds for $g(k,q)$ are known. 

In Section 3, we relate Waring numbers $g(k,q)$ with certain Cayley graphs of the form 
$$\G(k,q) = Cay(\ff_q, R_k) \quad \text{with} \quad R_k=\{x^k : x \in \ff_q^*\},$$ 
the so called generalized Paley graphs (GP-graphs, for short).
In Lemma 3.1 we give conditions on $\Gamma(k,q)$ to be connected, or equivalently, for $g(k,q)$ to exist. 
In Theorem \ref{gdiam} we show that, when exists, $g(k,q)$ equals the diameter of $\Gamma(k,q)$. 

In the following section, we consider GP-graphs which are Hamming graphs $H(b,q)$, since these graphs have known diameter $b$.
In Theorem \ref{WN g=b}, for $q=p^a$, we prove that 
$$g(\tfrac{q^{b-1}+\cdots +q+1}{b}, q^b)=b$$ 
for every $b>1$, provided that $b \mid q^{b-1}+\cdots +q+1$. 
Then, we use this result to show in Proposition~\ref{coro sobre} that every positive integer is the Waring number of some pair $(k,q)$.
Section~5 is somehow technical and can be skipped at first reading. We give necessary and sufficient arithmetical conditions for $b$ to 
divide $q^{b-1}+\cdots +q+1$. 
We will use this conditions to find infinite families of pairs $k,q$ for which we can give the exact value $g(k,q)$.

Section 6 contains the main results in the paper. Given a prime $p$ and integers $a, b$ where $b$ is coprime with $p$, in Theorem 6.1 we give simple conditions on $b$ and $x=p^a$ such that 
\begin{equation} \label{fla g=b}
g(\tfrac{p^{ab}-1}{b(p^a-1)},p^{ab})=b
\end{equation}
holds. 
Special cases are given in Corollaries 6.2--6.4 and 6.8--6.10. In Proposition 6.11 we show that if $b$ is an integer coprime with a prime $p$ and $a$ is a multiple of $\varphi(rad(b))$ then 
\eqref{fla g=b} holds. In particular, if $b=r^t$ with $r$ prime, formula \eqref{fla g=b} holds 
provided $r-1 \mid a$ (see Corollary 6.13). 

Finally, in Section 7, we give a new explicit lower bound for $g(k,p)$ with $p$ an odd prime. 
There are only two known lower bounds (to our best knowledge) for $g(k,q)$, one of them for $q=p$ prime. 
By using circulant GP-graphs (those associated to $\ff_p$) and using a known bound for the diameter of these graphs, 
for a prime $p\equiv 1 \pmod{2h}$ we get (see Proposition \ref{lower bound})
$$g(\tfrac{p-1}{2h},p) \ge \tfrac 12 \sqrt[h]{h! \, p} - \tfrac{h+1}2.$$
This bound only depends on the parameters $p$ and $h$ (compare with \eqref{lower 1} below).

\section{Summary of known results}
Here we briefly summarize the most important facts on the numbers $g(k,q)$. 
See \cite{MP} for more information.
Here, let $q=p^m$ with $p$ prime and $m$ a natural number and put $n=\frac{q-1}k$. 
We divide the results into two categories: exact values and bounds. We list the results chronologically.

\subsection{Exact values}
There are some results on exact values of Waring numbers.

\subsubsection*{Exact values for $g(k,p)$ with $p$ prime}
\begin{enumerate}[$(a)$]
	\item We have $g(k,p) = k$ for $k=1$, $2$, $\frac{p-1}2$ and $p-1$ % for $p$ prime 
	(1813, \cite{C}; 1977, \cite{Sm}). \sk 
	
	\item Small computed some exact values of $g(k,17)$ and $g(k,2)=1$ for any $k$ and the values for $g(3,p)$ for $p$ prime. 
Namely, $g(3,7)=3$, $g(3,p)=2$ for $p\equiv 1 \pmod 3$ and $p\ne 7$ and $g(3,p)=1$ for all other primes (1977, \cite{Sm}). \sk 

	\item Moreno and Castro gave some conditions depending on the $p$-adic weight of $k$ and the field size $q$ to have 
$g(k,q)=2$ or $g(k,q)\le 2$ (2003, \cite{MC}).  \sk 
	
	\item If $a,\, b$ are the unique positive integers with $a>b$ such that $a^2+b^2+ab =p$ then (2007, \cite{CiCP})
		$$g(k,p) = \begin{cases} 
	   a+b-1 & \quad \text{for } n=3, \\[1.5mm] 
	   \lfloor \frac 23 a + \frac 13 b\rfloor & \quad \text{for } n=6.
	\end{cases}$$

	\item If $a,\, b$ are the unique positive integers with $a>b$ such that $a^2+b^2 = p$, 
	then $g(k, p) = a-1$ for $n = 4$ (2007, \cite{CiCP}). 
\end{enumerate}

\subsubsection*{Exact values for $g(k,p^m)$ with $m>1$}
\begin{enumerate}[$(a)$]
\item If $2 \le k< \sqrt[4]q +1$ then $g(k,q)=2$ (1977, \cite{Sm2}). \sk

\item If $k\mid p^{\ell}+1$ then $g(k,p^{2\ell s})=2$ for $s\ne 1$ (2008, \cite{MC2}, for $k\ne p^{\ell}+1$; 2018, \cite{PV}, for $k=p^{\ell}+1$). \sk 
	
\item Let $p,\,r$ be primes such that $p$ is a primitive root modulo $r^m$ for some $m \in \N$. Then 
	\begin{equation} \label{ex fla 1}
	g \big( \tfrac{p^{\varphi(r^m)}-1}{r^m},p^{\varphi(r^m)} \big) = \tfrac 12 (p-1) \varphi(r^m)
	\end{equation} 
	where $\varphi$ denotes the Euler's totient function. 
	If in addition $p$ and $r$ are odd, then
	\begin{equation} \label{ex fla 2}
    g \big( \tfrac{p^{\varphi(r^m)}-1}{2r^m},p^{\varphi(r^m)} \big) =
	\begin{cases}
	r^{m-1}\lfloor \tfrac{pr}4 -\tfrac{p}{4r}\rfloor & \qquad \text{ if }	r<p,		\\[1.5mm]
	r^{m-1}\lfloor \tfrac{pr}4 -\tfrac{r}{4p}\rfloor & \qquad \text{ if } r\ge p. 
	\end{cases}
	\end{equation} 
 (2010, \cite{WW} for the case $m=1$ and \cite{KK} for the general case).	
\end{enumerate}

\subsection{Upper bounds} 
There are many upper bounds in the literature (for more bounds we refer for instance to \cite{CMS} or \cite{MP}). These are the most typical results on Waring numbers. We divide them into two cases, $q=p$ prime and $q$ a prime power.

\subsubsection*{Upper bounds for $g(k,p)$ with $p$ prime}
Here $n=\tfrac{p-1}{k}$.
\begin{enumerate}[$(a)$]
\item $g(k,p)\le k$ with equality for $k=1,\,2,\, \frac{p-1}2,\, p-1$ (1813, \cite{C}). \sk 

\item $g(k,p) \le [\tfrac k2]+1$ for any $n>2$ and $k>1$ (1959, \cite{CMS}). \sk 

\item If $k$ is a proper divisor of $p-1$ and $k\ge (q-1)^{\frac 47}$, 
then $g(k,p) \le 170\frac{k^{7/3}}{(p-1)^{4/3}} \ln p$ (1988, \cite{GV}). \sk 

\item If $k< \frac{p(\ln \ln p)^{1-\varepsilon}}{\ln p}$ then $g(k,p)\ll(\ln k)^{2+\varepsilon}$ (1994, \cite{Kon}). \sk 

\item For any $\varepsilon > 0$ there exists a constant $c_{\varepsilon}$ such that for any $k\ge 2$ and $p \ge \frac{k\ln k}{(\ln (\ln k+1))^{1-\varepsilon}}$, then $g(k,p)\le c_{\varepsilon} (\ln k)^{2+\varepsilon}$ (1994, \cite{Kon}). \sk 

\item If $n\ge 2$, then $g(k,p)\le 83 \sqrt k$ (2008, \cite{CP}). \sk 

\item If $\varphi(n)> \ell$ for some positive integer $\ell$ then $g(k,p)\le c_\ell \, \sqrt[\ell]k$ for some constant $c_\ell$
(2009, \cite{CiCP}).
\end{enumerate}

\subsubsection*{Upper bounds for $g(k,p^m)$ with $m>1$}
\begin{enumerate}[$(a)$]
\item If $q^{s-1}>(k-1)^{2s}$ then $g(k,q)\le s$ (1998, \cite{Win}). \sk 

\item If $d=\frac{p-1}{(m,p-1)}$ then $g(k,q) \le m \cdot  g(d,p)$ (1998, \cite{Win}). \sk 

\item If $q^{\frac 37}+1\le k\le q^{\frac 12}$, then $g(k,q)\le 8$ (2009, \cite{Ci}). \sk 

\item If $k< \sqrt q$ then $g(k,q)\le 8$ (2009, \cite{GlR}). \sk 

\item $g(k,p^2) \le 16 \sqrt{k+1}$ and  $g(k,p^n) \le 10 \sqrt{k+1}$ for $n\ge 3$ (2009, \cite{Ci}). \sk 

\item For any $\varepsilon> 0$, $k\le q^{1-\varepsilon}$ and if $\ff_q = \ff_p(x^k)$ for some $x\in \ff_q$, there is a constant
$c(\varepsilon)$ such that $g(k,q)\le c(\varepsilon)$ (2011, \cite{Gl}).   
\end{enumerate}

\subsection{Lower bounds}
To our knowledge, there are only three lower bounds for Waring numbers $g(k,q)$, and in two cases $q=p$ is prime.

\begin{enumerate}[$(a)$]
\item If $p$ is a prime with of the form $p=kn+1$, 
%$p\equiv 1 \pmod n$ for some $n$, 
then 
\begin{equation} \label{lower 1}
g(\tfrac{p-1}{n},p) \ge \frac{\sqrt[\varphi(n)]{p}-1}{2 c_n}
\end{equation} 
where $c_n$ is a constant depending only on $n$ and $\varphi$ is the Euler function (1993, \cite{GS}). \sk 

\item  If $r$ is prime and $p$ is a primitive root of unity modulo $r$ then (2001, \cite{Win2})
\begin{equation} \label{lower 2} 
g(\tfrac{p^{r-1}-1}{r},p^{r-1}) \ge \tfrac 12 (\sqrt[r-1]{rk}-1).  
\end{equation} 

\item If $p$ is a prime and $k=\frac{p-1}{n}$ then  
\begin{equation} \label{lower 4}
g(\tfrac{p-1}n,p) \ge \frac{(1-\tfrac 1p)}{2C_n} \, \sqrt[\varphi(n)]{p}  \gg \frac{1}{\sqrt{\log n}} \sqrt[\varphi(n)]{p} 
\end{equation}
with 
\begin{equation} \label{Cn}
C_n= \prod_{\ell \, \mid \, n} \ell^{\frac{1}{2(\ell-1)}},
\end{equation}
where the product is over all odd prime numbers $\ell \mid n$ (2007, \cite{Ci}).

\end{enumerate}

\section{Generalized Paley graphs and the Waring's problem}
If $G$ is a group and $S$ is a subset of $G$ not containing $0$, the Cayley graph $\Gamma = X(G,S)$ is the digraph with 
vertex set $G$ and where two vertices $u,v$ form a directed edge from $v$ to $u$ in $\Gamma$ if and only if $v-u \in S$. 
If $S$ is symmetric ($S=-S$), then $X(G,S)$ is a simple (undirected) graph. 

Let $q=p^m$ with $p$ a prime number and $k$ a non-negative integer with $k\mid q-1$. The \textit{generalized Paley graph} 
(or GP-graph for short) is the Cayley graph
\begin{equation} \label{Gammas}
\G(k,q) = X(\ff_{q},R_{k}) \quad \text{with } \quad R_{k} = \{ x^{k} : x \in \ff_{q}^*\}.
\end{equation} 
That is, $\G(k,q)$ is the graph with vertex set $\ff_{q}$ and two vertices $u,v \in \ff_{q}$ are neighbors if and only if 
$v-u=x^k$ for some $x\in \ff_{q}^*$. 
These graphs are denoted by $GP(q,\frac{q-1}k)$ in \cite{LP}.

Notice that if $\omega$ is a primitive element of $\ff_{q}$, then $R_{k} = \langle \omega^{k} \rangle$, and this implies that $\G(k,q)$ 
is a $\frac{q-1}{k}$-regular graph. 
Assuming that $k \mid \tfrac{q-1}2$ if $p$ is odd one has that $\G(k,q)$ is a simple graph. 
For $p=2$, the first condition is not necessary since in this case the graph $\G(k,q)$ is always undirected.
When $k=1$ we get the complete graph $\G(1,q)=K_q$ and when $k=2$ we get the classic Paley graph $\Gamma(2,q) = P(q)$. 

\sk

The graph $\G$ is connected if for any pair of 
vertices $u,v$ there is a walk from $u$ to $v$. In this case, we denote by $d(u,v)$ the distance between $u$ and $v$, i.e.\@ the minimum length of a walk from $u$ to $v$.

We collect algebraic, spectral and arithmetic conditions for a GP-graph to be connected.
\begin{lem}\label{equiv conn}
Let $\G(k,q)$ with $q=p^m$ and $n=\frac{q-1}{k}$. The following are equivalent:
\begin{enumerate}[$(a)$]
\item $\G(k,q)$ is connected. \sk 

\item $R_k$ additively generates $\ff_q$. \sk 

\item The Waring number $g(k,q)$ exists. \sk 

\item $n$ is an eigenvalue of $\Gamma$ with multiplicity $1$. \sk 

\item $n$ is a primitive divisor of $q-1$. \sk 

\item $\frac{q-1}{p^a-1}\nmid k$ for all $1\le a\le m-1$. \sk 
\end{enumerate}
\end{lem}

Note that $(e)$ says that $g(k,p)$ always exists for $k\mid p-1$ and $p$ prime.

\begin{proof}
The equivalence ($a$) $\Leftrightarrow$ ($b$) is known for general Cayley graphs and ($b$) $\Leftrightarrow$ ($c$) is clear from
the definitions. From spectral graph theory one knows that ($a$) $\Leftrightarrow$ ($d$). 
Recall that $n$ 
is a primitive divisor of $p^m-1$ if $n$ does not divide $p^t-1$ for every $t < m$. 
Finally,  
($a$) $\Leftrightarrow$ ($e$) follows by the definition of GP-graphs and ($e$) $\Leftrightarrow$ ($f$) is straightforward.
\end{proof}

Recall that a walk of length $s$ from $u$ to $v$ in $\Gamma$ is a sequence of edges $e_1,\ldots,e_s$ of $\Gamma$ such that 
$e_1=u v_1$, $e_{2}=v_{1}v_2, \ldots, e_{i}=v_{i-1}v_i$, $\ldots, e_{s}=v_{s-1}v$.

\begin{lem} \label{swalks}
There is a walk of length $s$ in $\G(k,q)$ from $u$ to $v$ if and only if there exist $x_1,\ldots,x_s \in \ff_q^*$ such that 
$v-u = x_1^k +\cdots + x_s^k$.
\end{lem}

\begin{proof}
If $uw$ and $wv$ are edges of $\G(k,p)$, then $v-w=x_1^k$ and $w-u=x_2^k$ for some $x_1,x_2 \in \ff_{q}^n$. Thus, 
$v-u =(v-w)+(w-u)= x_1^k + x_2^k$. Conversely, given $v-u = x_1^k + x_2^k$ we take $w=u+x_2^k$. Clearly $w$ is a vertex, $uw$ and $wv$ are edges of $\G$ and hence there is a 2-walk from $u$ to $v$. The general result follows by induction.  
\end{proof}

We now give a general strategy that can be used to compute Waring numbers. 
We will relate the number $g(k,q)$ with the diameter of the graph $\G(k,q)$.
This was previously used by Hamidoune in the case $q=p$ prime (as mentioned by Garc\'ia-Sol\'e in (\cite{GS}).
Recall that the \textit{diameter} of $\Gamma$ is its maximal eccentricity, given by 
$$\delta(\Gamma) =  \max_{u\in \ff_q} \epsilon(u) = \max_{u\in \ff_q} \max_{v\in \ff_q} d(u,v).$$
That is, the diameter of $\Gamma$ is the greatest length between all minimal walks in $\Gamma$.

\begin{thm} \label{gdiam}
If the GP-graph $\G=\G(k,q)$ is connected then 
\begin{equation} \label{g=d} 
%r(\Gamma) \le 
g(k,q) = \delta(\G).
\end{equation} 
\end{thm}

\begin{proof}
We first show that the diameter can be realized with paths starting from $0$.
Let $u,v$ be vertices of $\G$ such that $d(u,v)=\delta(\G)=\delta$. 
Clearly, $\delta \ge d(0,c)$ for all $c\in\ff_{q}^*$. In particular, we have that $\delta\ge d(v-u,0)$. 
Assume that $\delta > d(v-u,0)=t$. Then, there is a sequence $x_{1}, \ldots, x_{t} \in \ff_{q}^*$ such that 
$v-u=x_{1}^k+\cdots+x_{t}^k$, by Lemma \ref{swalks}. This induces a walk from $u$ to $v$, i.e.\@ $d(u,v) \le t < \delta = d(u,v)$, which is absurd. Therefore $\delta = d(0,v-u)$, and we have that 
\begin{equation}\label{diam gama}
\delta(\G) = \max\{d(0,c): c\in \ff_{q}^*\}.
\end{equation}

Notice that every element of $\ff_{q}^*$ can be written as a sum of $\delta$ $k$-th powers. In fact, if 
$c\in\ff_{q}^*$ with $d(0,c) = s'$ then $s' \le \delta$, then there exist $y_{1},\ldots,y_{s'}$ such that 
$y_1^k + \cdots + y_{s'}^k = c$. Defining $y_{s'+1} = \cdots = y_{\delta}=0$ 
we obtain that $y_1^k + \cdots + y_{\delta}^k = c$ as desired.

Now, by \eqref{diam gama}, there exists $a\in\ff_{q}^*$ such that $\delta(\G)=d(0,a)$. 
Clearly, $a$ cannot be written as a sum of less than $\delta$ $k$-th powers, otherwise we obtain a walk from $0$ to $a$
in $\G$ with length less than $d(0,a)$. Therefore, $g(k,q) = \delta(\G)$ as claimed.
\end{proof}

\begin{rem}
($i$) By definition of radius of a graph and by \eqref{g=d} we have 
$$ r(\G(k,q)) = \min\limits_{u\in \ff_q} \, \max_{v\in \ff_q} \, d(u,v) \le g(k,q). $$ 

\noindent ($ii$) 
The previous theorem can be applied whether the graph $\G(k,q)$ is directed or not.
\end{rem}

\section{Exact values for $g(k,q)$ through Hamming GP-graphs}
Here we give a general expression for $g(k,q)$ using GP-graphs which are also Hamming graphs, since they have known diameter. From this result, we will show that every positive integer is the Waring number for a certain pair $(k,q)$. 

A Hamming graph $H(b,q)$ is a graph with vertex set $V=K^b$ where $K$ is any set of size $q$
(typically $\ff_q$ in applications), and where two $b$-tuples form and edge if and only if they differ in exactly one coordinate. 
Clearly, $H(b,q)$ is a connected graph with diameter $b$.

From now on, we will adopt the following notation. For a positive integer $b$ we put  
\begin{equation} \label{Psi b}
\Psi_b (x) = \tfrac{x^b-1}{x-1} = x^{b-1} + \cdots + x +1. 
\end{equation}

\begin{thm} \label{WN g=b}
Let $p$ be a prime and let $a, b$ be positive integers. 
If $b \mid \Psi_b(p^{a})$ then 
\begin{equation} \label{g=b}
g \big( \tfrac{p^{ab}-1}{b(p^{a}-1)}, p^{ab} \big) = g \big(\tfrac 1b \Psi_b(p^{a}), p^{ab} \big) = b.
\end{equation} 
\end{thm}

\begin{proof}
Let $m=ab$. If $b=1$, we clearly have $g(1,p^m)=1$ and hence \eqref{g=b} holds. Thus, suppose $b>1$.
Note that $\tfrac{p^m-1}{p^{a}-1} = \Psi_b(p^{a})$.
Hence, $k=\tfrac{p^m-1}{b(p^{a}-1)}$ is an integer if and only if $b \mid \Psi_b(p^{a})$. 

Now, those GP-graphs which are Hamming graphs are characterized in \cite{LP}. 
In fact, $\G(\frac{p^{m}-1}n,p^m)$ is Hamming if and only if 
$$n = b(p^{a}-1)$$ 
for some $b\mid m$ with $b>1$. 
In this case we have 
$$\G(\tfrac 1b \Psi_b(p^a), p^{ab}) = \G(\tfrac{p^{m}-1}n, p^m) \cong H(b,p^{a})$$ 
(see the proof of Theorem 4.1 (2) in \cite{LP}, \S 3) and therefore, 
by Theorem \ref{gdiam}, we have 
$$g \big( \tfrac 1b \Psi_b(p^a),\,p^m \big) = \delta \big(\G(\tfrac 1b \Psi_b(p^a), p^m) \big) = 
\delta (H(b, p^{a})) = b,$$
as we wanted to see. 
\end{proof}

If $b$ is prime, $\Psi_b(x)$ equals the $b^{th}$ cyclotomic polynomial $\Phi_b(x)$. In general, we have 
\begin{equation} \label{Psi factors} 
\Psi_b(x) = \prod_{d\mid b, \, d > 1} \Phi_d(x).
\end{equation} 

We now give some easy sufficient conditions assuring that $\tfrac 1b \Psi_b(x)$ is an integer.
\begin{lem} \label{lema conds}
	Let $x$ and $b$ be integers. Thus, we have:
	\begin{enumerate}[$(a)$]
		\item If $x \equiv 1 \pmod b$ or else $b$ is even and $x \equiv -1 \pmod b$ then $b \mid \Psi_b(x)$. \sk
		
		\item If every prime divisor $b$ divides $\Phi_d(x)$ for some $d\mid b$ with $d>1$ then $b \mid \Psi_b(x)$.
	\end{enumerate}
\end{lem}

\begin{proof}
For ($a$), since $x\equiv 1 \pmod b$ we have 
	$\Psi_b(x)=x^{b-1} + \cdots  + x + 1\equiv b \equiv 0 \pmod b$, as we wanted. 
	Similarly, if $x\equiv -1 \pmod b$ we have 
	$x^{b-1} + \cdots + x + 1 \equiv (-1+1)+ \cdots +(-1+1) \equiv 0 \pmod b$, 
	provided that $b$ is even. 
	Item ($b$) follows directly from \eqref{Psi b} and \eqref{Psi factors}.
\end{proof}

As a direct consequence of Theorem \ref{WN g=b} and Lemma \ref{lema conds}, we have the following conditions ensuring that \eqref{g=b} holds.
\begin{coro} \label{conditions}
Let $p$ be prime and let $a$ and $b$ be positive integers. 
\begin{enumerate}[$(a)$]
\item If $p^a \equiv 1 \pmod b$ or else $b$ is even and $p^a \equiv -1 \pmod b$ then \eqref{g=b} holds. \sk

\item If every prime divisor of $b$ divides $\Phi_d(p^a)$ for some $d\mid b$ with $d>1$ then \eqref{g=b} holds.
\end{enumerate}
\end{coro}

Note that if $p$ is a prime and $a,b $ are positive integers, then %$n_b = b(p^{\frac mb}-1)$ is 
$b$ does not necessarily divides $\Psi_b(p^{a})$, as the next example shows. 

\begin{exam}
Let $p=2$ and $m=12$. The positive divisors of $12$ are $\{1, 2, 3, 4, 6, 12\}$.
Then, $b=2,4,6$ and $12$ 
do not divide $\Psi_b(2^{\frac{12}{b}}) = \frac{2^{12}-1}{2^{\frac{12}{b}}-1}$ since $b \nmid 2^{12}-1$, the last number being odd. 
Thus, we cannot apply Theorem \ref{WN g=b} in the case when $b$ is even. 
However, since $b=3$ divides $\Psi_3(2^3) = \frac{2^{12}-1}{2^{4}-1}=2^8+2^4+1=273$, 
by Theorem~\ref{WN g=b} we get 
$$g(91, 4{.}096) = g(\tfrac{2^8+2^4+1}3, 2^{12}) = 3.$$ 

Similarly, taking $m=6$, the only divisor of $2^6-1 =7\cdot 9$ of the form $b(2^{a}-1)$ is 
$3(2^{2}-1)=9$ which by \eqref{g=b} gives 
$g(7,64) = g(\tfrac{2^4+2^2+1}3 , 2^6) = 3$. 
\end{exam}

We now show that \eqref{g=b} can also hold for $p^a \equiv c \pmod b$ with $c\ne \pm 1$ (see ($a$) in Corollary~\ref{conditions}).

\begin{exam} \label{pow exam}
Let $b=9$, $m=9a$ for some $a$ and $p$ any prime such that $p^a \equiv 7 \pmod 9$. Let $x=p^a$. 
Since the order of $x$ modulo $9$ is $3$, we have 
$$ \Psi_9(x)= x^8 +\cdots + x^1 +1 \equiv 3 (x^2 + x^1 + x^0) \equiv 3(49+7+1) = 3\cdot 57 \equiv 0 \pmod 9.$$ 
By Theorem \ref{WN g=b}, we obtain that
$$g (\tfrac 19 \Psi_9(x), p^{9a}) = g (\tfrac{p^{8a}+ \cdots + p^a + 1}{9}, p^{9a}) = 9 $$
for any $p^a  \equiv 7 \pmod 9$.
For instance, if we take $p=7$ and $a=1$, we have  
$$g(747{.}289, 40{.}353{.}607) = %g(\tfrac 19 \Psi_9(7), 7^9) = 
g (\tfrac{7^9-1}{9(7-1)}, 7^9) = 9.$$
Also, we can take $p=11$ and $a=4$, since $11^4\equiv 7 \pmod 9$. Hence we have
\begin{align*}
\tfrac{11^{36}-1}{90} & = 343{.}474{.}228{.}143{.}007{.}473{.}729{.}703{.}921{.}520{.}971{.}704, \\ 
11^{36} & =30{.}912{.}680{.}532{.}870{.}672{.}635{.}673{.}352{.}936{.}887{.}453{.}361.
\end{align*}
and $g (\tfrac{11^{36}-1}{90}, 11^{36}) = 9$.
\end{exam}

\subsection*{The Waring function}
We now define the Waring function from Waring pairs. 
\begin{defi}
We say that a pair of positive integers $(k,q)$, such that $q$ is a prime power and $k\mid q-1$, is a \textit{Waring pair} if $g(k,q)$ exists. We denote by $\mathbb{W}$ the set of all such pairs. 
Consider the \textit{Waring function} sending every Waring pair to the corresponding Waring number, i.e.\@
\begin{equation} \label{g-function}
g: \mathbb{W} \subset \N \times \N \rightarrow \N, \qquad (k,q) \mapsto g(k,q).
\end{equation}
\end{defi}

We now show that $g$ is surjective.
That is, every positive integer number is the Waring number $g(k,q)$ for some pair $(k,q)$ in some (generically non-prime) 
finite field. 

\begin{prop} \label{coro sobre}
Let $b \in \N$. Then, there exist $k \in \N$ and a prime power $q=p^m$ for some $m \ge 1$ 
$(m>1$ if $b>1$$)$ such that $g(k,q)=b$. Moreover, if $b$ is odd one can take 
$k = \frac{2^{b\varphi(b)}-1}{b(2^{\varphi(b)}-1)}$ and $q=2^{b\varphi(b)}$, that is  
\begin{equation} \label{gkq=b} 
g \big(\tfrac{2^{b\varphi(b)}-1}{b(2^{\varphi(b)}-1)}, 2^{b\varphi(b)} \big) = b.
\end{equation}
\end{prop}

\begin{proof}
Let $p$ be a prime number that is coprime with $b$. Thus, $p$ is a unit in $\mathbb{Z}_b$. Hence, there exists some positive integer $a$ such that $p^a \equiv 1 \pmod b$. Let $m=ab$ and $k=\tfrac{1}{b}\Psi(p^a)$, which is an integer by Lemma \ref{lema conds}. Hence $k \mid p^m-1$. Thus, taking $q=p^m$, we know that $g(k,q)$ exists and, by Theorem \ref{WN g=b}, we further have that 
$g(k,q)=b$, as desired. 

If $b$ is odd, we choose $p=2$ and $a=\varphi(b)$. By Euler's Theorem $2^{\varphi(b)} \equiv 1 \pmod b$ and then
\eqref{gkq=b} holds by Corollary \ref{conditions}.
\end{proof}

The numbers provided by the proposition grow rapidly. For instance, for the first odd numbers, \eqref{gkq=b} give
\begin{align}
\begin{aligned}
& 3 = g \big( \tfrac{2^6-1}{3(2^2-1)}, 2^{3\cdot 2} \big) = g(7,64), \\
& 5 = g \big( \tfrac{2^{20}-1}{5(2^4-1)}, 2^{5\cdot 4} \big) = g(13{.}981, 1{.}048{.}576), \\
& 7 = g \big( \tfrac{2^{42}-1}{7(2^6-1)}, 2^{7\cdot 6} \big) = g(9{.}972{.}894{.}583, 4{.}398{.}046{.}511{.}104), \\
& 9 = g \big( \tfrac{2^{54}-1}{9(2^6-1)}, 2^{9\cdot 6} \big) = g(31{.}771{.}425{.}942{.}649, 18{.}014{.}398{.}509{.}481{.}984).
\end{aligned}
\end{align}

We close the section with some questions. 
\subsection*{Questions}
Given $b$, does $g$ take the value $b$ infinitely many times?
If this is the case, can we find a tower of fields $\{F_i\}$ and a sequence of integers $k_i$ such that the associated Waring numbers are the same, i.e.\@ $g(k_i, |F_i|)=b$?

\section{Integrality of $\frac{1}{b} \Psi_{b}(x)$}
In this section, we give necessary and sufficient conditions for positive coprime integers $b$ and $x$ to have $\frac 1b \Psi_{b}(x) \in \Z$. 
This will allow us to apply %the main 
Theorem~\ref{WN g=b} in different contexts.
We denote by $ord_b(x)$ the order of $x$ modulo $b$, i.e.\@ 
the least positive integer $s$ such that $x^s \equiv 1 \pmod b$.
We consider the cases when $b$ is squarefree, a prime power or the general case separately.

\subsubsection*{The case $b$ is squarefree}

\begin{lem} \label{x1 mod r}
	Let  $x$ be an integer coprime with $b = r_1\cdots r_\ell$ with $r_1 <  \cdots < r_\ell$ primes. Then,
	\begin{equation}\label{div b Psi}
	b\mid \Psi_b(x) \quad \Leftrightarrow  \quad x\equiv 1 \!\pmod{r_1} \quad \text{and} \quad x^{b/{r_i}} \equiv 1 \! \pmod{r_i}, \quad i=2, \ldots, \ell.
	\end{equation}
In particular, if $r,r'$ are distinct primes, both coprime with $x$, we have the following: % cases of interest:
	\begin{enumerate}[$(a)$]
		\item $r \mid \Psi_r(x)$ if and only if $x \equiv 1 \pmod r$. \sk 
	
		\item If $r$ and $x$ are odd then $2r \mid \Psi_{2r}(x)$ if and only if $x \equiv \pm 1 \pmod r$. \sk 
		
		\item If $r,r'$ are odd and $r \nmid r'-1$ then $rr' \mid \Psi_{rr'}(x)$ if and only if $x\equiv 1 \pmod{rr'}$.
	\end{enumerate}
\end{lem}

\begin{proof}
Clearly, $b \mid \Psi_b(x)$ if and only if $r_{i} \mid \Psi_{b}(x)$ for all $i=1,\ldots,\ell$.
Thus, we will prove that $r_i \mid \Psi_{b}(x)$ for $i=1,\ldots,\ell$ if and only if $x\equiv 1 \pmod{r_1}$ and 
$x^{b/{r_i}} \equiv 1 \pmod{r_i}$ for $i=2,\ldots,\ell$. 	
	
We first show that $r_{1} \mid \Psi_{b}(x)$ if and only if $x\equiv 1 \pmod{r_1}$. 
Suppose that $x\equiv 1 \pmod{r_1}$. As in the proof of Lemma \ref{lema conds}, we have $\Psi_b(x) \equiv b \pmod{r_1}$. Thus, $r_1 \mid \Psi_{b}(x)$ since $r_1 \mid b$. For the converse, suppose that	$r_{1} \mid \Psi_{b}(x)$ and $x \not\equiv 1 \pmod{r_1}$. Then, we have that
\begin{equation} \label{Psi equiv}
\Psi_{b}(x)(x-1) = x^b-1\equiv x^{b/{r_i}}-1 \pmod{r_i}
\end{equation} 
with $i=1$, since $x^{r_1}\equiv x \pmod{r_1}$ by Fermat's theorem. 
This implies that $x^{b/{r_1}} \equiv 1 \pmod{r_1}$, since $r_1 \mid \Psi_b(x)$. Thus, some prime 
factor $p$ of order $ord_{r_1}(x)$ divides $b/r_1$ and hence $p=r_j$ for some $j=2,\ldots, \ell$. 
On the other hand, $p \le ord_{r_1}(x) \le r_1-1 < r_1 < r_j$ for all $j=2,\ldots,\ell$. This is a contradiction and hence
$x\equiv 1 \pmod{r_1}$.
	
For $i = 2, \ldots, \ell$, we will show that $r_i \mid \Psi_{b}(x)$ if and only if $x^{b/r_i} \equiv 1 \pmod{r_i}$.  
As before, $x\equiv 1 \pmod{r_i}$ implies that $r_i \mid \Psi_{b}(x)$. 
Now, if $x^{b/r_i} \equiv 1 \pmod{r_i}$ and $x \not \equiv 1 \pmod{r_i}$, 
we have that \eqref{Psi equiv} holds for $i=2,\ldots,\ell$.
Thus $\Psi_{b}(x)\equiv 0 \pmod{r_i}$ since $x-1 \not \equiv 0 \pmod{r_i}$ and $\mathbb{Z}_{r_i}$ has no zero divisors.
Now, if $r_i \mid \Psi_{b}(x)$, then $x^{b/r_i}-1 \equiv 0 \pmod{r_i}$, by \eqref{Psi equiv}, as desired.

\smallskip
	
We now check the cases in ($a$)--($c$). Clearly, ($a$) is a consequence of \eqref{div b Psi} taking $h=1$. 
In case ($b$), we have $b=2r$. Thus, $b\mid \Psi_b(x)$ if and only if $x\equiv 1 \pmod 2$ and $x^2\equiv 1 \pmod r$. Since $(x,b)=1$, the statement $x\equiv 1 \pmod 2$ and $x^2\equiv 1 \pmod r$ is equivalent to $x\equiv \pm1 \pmod r$.
Finally, for ($c$), suppose that $b=r_1 r_2$ with $r_1 \nmid r_2-1$. 
By \eqref{div b Psi}, we have that $x^{r_1}\equiv 1 \pmod{r_2}$. Then, $ord_{r_2}(x) \mid r_1$ and this implies that 
$ord_{r_2}(x)=1$ or $ord_{r_2}(x)=r_1$. Notice that $ord_{r_2}(x)$ cannot be $r_1$, since if $ord_{r_2}(x)=r_1$. Hence, 
$r_1\mid r_2-1$ by Lagrange's theorem, which contradicts our assumption. Therefore $ord_{r_2}(x)=1$. That is, $x\equiv 1 \pmod{r_2}$, and by the Chinese remainder's theorem we have that $x\equiv 1 \pmod b$, as desired.
\end{proof}

\subsubsection*{The case $b$ is a prime power}
\begin{lem} \label{Lem r pow}
Let $r$ be a prime and $x \in \N$ such that $(x,r)=1$. For any $t \in\N$ we have
	\begin{equation}\label{div b Psi pow2}
	r^t\mid \Psi_{r^t}(x) \qquad \Leftrightarrow  \qquad ord_{r^t}(x)=r^h \quad 
	\text{ for some } \quad 0\le h \le t-1.
	\end{equation}
\end{lem}

\begin{proof}
	Suppose that $r^t\mid \Psi_{r^t}(x)$. Since $x$ is coprime with $r$ we have $ord_{r^t}(x)\mid \varphi(r^t)=r^{t-1}(r-1)$ and hence 
	$\kappa=\frac{\varphi(r^t)}{ord_{r^t}(x)}\in \mathbb{Z}$. Thus, we can put $ord_{r^t}(x)=r^{h}s$ for some $s\mid r-1$ and $h\le t-1$. It suffices to show that $s=1$.
	Notice that 
	$$\Psi_{r^t}(x)=1+\sum_{j=1}^{r^t-r^{t-1}}x^j+\sum_{j=r^t-r^{t-1}+1}^{r^t-1}x^j.$$ 
	By modularity we have 
	$$ \sum_{j=1}^{r^t-r^{t-1}}x^j  \equiv \kappa \sum_{j=1}^{ord_{r^t}(x)}x^j \pmod{r^t}
	\qquad \text{and} \qquad 
\sum_{j=r^t-r^{t-1}+1}^{r^t-1}x^j  \equiv \Psi_{r^{t-1}}(x)-1 \pmod{r^t}.$$
	This implies that
	\begin{equation}\label{Psi recurs}
	\Psi_{r^t}(x)\equiv \kappa \sum_{j=1}^{ord_{r^t}(x)}x^j+\Psi_{r^{t-1}}(x)\pmod{r^t}.
	\end{equation}
By hypothesis we have $\Psi_{r^t}(x)\equiv 0 \pmod {r^t}$. 
Thus, multiplying \eqref{Psi recurs} by $x-1$ we get 
$$0 \equiv  \kappa (x-1) \sum_{j=1}^{ord_{r^t}(x)}  x^j + (x-1) \Psi_{r^{t-1}}(x) \equiv x^{r^{t-1}}-1 \pmod{r^t}$$
since  $\sum_{j=1}^{ord_{r^t}(x)} (x-1) x^j = x^{ord_{r^t}(x)+1} -x \equiv 0 \pmod {r^t}$.
In this way we have $x^{r^{t-1}}\equiv 1 \pmod{r^t}$. Therefore, $ord_{r^{t}}(x) \mid r^{t-1}$, that is $s=1$ as desired.

For the converse, assume that $ord_{r^t}(x)=r^h$ for some $h\le t-1$ with $r$ odd. We will do induction on $t$. If $t=1$, the statement follows from ($a$) of Lemma \ref{x1 mod r}.
	Suppose now that $t>1$ such that the statement holds for all $1\le t'<t$. 
	By modularity,
	$$\Psi_{r^t}(x)\equiv r^{t-h}\Psi_{r^h}(x) \pmod{r^t}.$$
	Notice that $x^i$ runs over the group $\langle x\rangle$ when $i$ runs over $\{0,1,\ldots,r^h-1\}$, then we have that
	$$\Psi_{r^h}(x)\equiv\sum_{y\in \langle x\rangle } y \pmod{r^t} .$$ 
Since $\langle x  \rangle$ is cyclic there is a unique subgroup  $\langle x^d  \rangle$ of order 
$\tfrac{\# \langle x  \rangle}{d}$ for each divisor $d$ of $\# \langle x  \rangle$.
Therefore, if we take another element of order $r^h$, then this element also generates $\langle x\rangle$. 
	
	Since $r$ is odd, we can choose an integer $\alpha$ such that $\langle \alpha\rangle= (\mathbb{Z}_{r^s})^*$ for all $s$.
	Now, taking $\beta=\alpha^{(r-1)r^{t-1-h}}$ we have $ord_{r^t}(\beta)=r^h$. Thus, $\langle x\rangle =\langle \beta\rangle$ in $(\mathbb{Z}_{r^t})^*$ and
	$$\Psi_{r^h}(x)\equiv \Psi_{r^h}(\beta) \pmod{r^t} .$$
	It is enough to prove that $r^h\mid \Psi_{r^h}(\beta)$. Taking into account that $t-1-h=h-1-(2h-t)$ and $\alpha$ is a primitive element, we obtain that
	$$ord_{r^h}(\beta)=
	\begin{cases} 
	1 & \qquad \text{if $2h\le t$}, \\[1mm]
	r^{2h-t} & \qquad \text{if $2h>t$}.
	\end{cases} $$ 
	Clearly, if $2h\le t$, then $\beta \equiv 1 \pmod{r^h}$ and thus $r^h\mid \Psi_{r^h}(\beta)$. 
	If $2h>t$, we can take $t'=h$ and $h'=2h-t$. Then, $ord_{r^{t'}}(\beta)=r^{h'}$ with $0\le h'\le t'-1$. By inductive hypothesis we get $r^{h}\mid\Psi_{r^h}(\beta)$, i.e.\@ $r^t\mid \Psi_{r^h}(x)$, as desired. 
	
	The case $r=2$ can be proved by induction in the same way as before by using that if $(x,2)=1$ then by 
	Euler's theorem we have that $ord_{2^s}(x)\mid \varphi(2^s)=2^{s-1}$ for all $s$, that is $ord_{2^s}(x)$ is a power of $2$ for all $s$, and the proof is complete.
\end{proof}

\subsubsection*{The general case, $b$ any positive integer}
\begin{lem}\label{gen case}
Let $x$ be an integer coprime with $b = r_1^{t_1} r_2^{t_2} \cdots r_\ell^{t_\ell}$ with $r_1 < r_2 < \cdots < r_\ell$ primes.
If $ord_{r_{i}^{t_i}}(x)=r_{i}^{h_i}$ with $0\le h_i\le t_{i}-1$ for all $i$, then $b \mid \Psi_{b}(x)$.
\end{lem}

\begin{proof}
		Clearly $b\mid\Psi_{b}(x)$ if and only if $r_{i}^{t_i}\mid \Psi_{b}(x)$ for all $i=1,\ldots,\ell$. 
		By hypothesis, there is some $h_i \in \{1,\ldots,t_i-1\}$ such that $ord_{r_{i}^{t_i}}(x)=r_{i}^{h_i}$ for any $i=1,\ldots,\ell$.
	 The previous proposition implies that $r_{i}^{t_i}\mid \Psi_{r_{i}^{t_i}}(x)$ for each $i$. 
	 Thus, if $b_i = b/r_i^{t_i}$, by modularity we get
		$$\Psi_{b}(x)= \Psi_{b_i r_{i}^{t_i}}(x)\equiv b_i \Psi_{r_{i}^{t_i}}(x)\equiv 0 \pmod{r_{i}^{t_i}}$$
		for every $i$. 
	Therefore $r_{i}^{t_i}\mid \Psi_{b}(x)$ for all $i=1,\ldots,\ell$ and hence $b\mid\Psi_{b}(x)$, as we wanted.
\end{proof}

\section{Exact values of $g(\frac 1b \Psi_{b}(q),q^b)$}
In this section we give sufficient (and in almost all cases necessary) conditions for \eqref{g=b} to hold. 
Then, we will obtain several families of exact values of Waring numbers. 
 
\begin{thm} \label{exact val Hamm}
	Let $a,b$ be positive integers and let $x=p^a$ with $p$ prime. Then, 
		\begin{equation} \label{gkqb}
			g \big( \tfrac{p^{ab}-1}{b(p^{a}-1)}, p^{ab} \big) = g \big(\tfrac 1b \Psi_b(p^{a}), p^{ab} \big) = b	
		\end{equation}
	holds in the following cases:
		\begin{enumerate}[$(a)$]
	\item If $b=r$ is a prime different from $p$ 
	and $x\equiv 1 \pmod r$. \sk 
	 
	\item If $b=2r$ with $r$ an odd prime,  $x$ coprime with $b$ and $x\equiv \pm1 \pmod r$. \sk 

	\item If $b=r r'$ with $r<r'$ odd primes such that $r \nmid r'-1$ and $x\equiv 1 \pmod{rr'}$. \sk 
	 
	\item If $b=r_1 r_2 \cdots r_\ell$ with $r_1 < r_2 < \cdots < r_\ell$ primes different from $p$ with $x\equiv 1 \pmod{r_1}$ and $x^{b/r_i} \equiv 1 \pmod{r_i}$ for $i=2, \ldots,\ell$. \sk 
%	 where $\hat{r_i} = \tfrac{b}{r_i}$.
	 
	\item If $b=r^t$ with $r$ prime such that $ord_{b}(x)=r^h$ for some $0\le h<t$. \sk 
	
	\item If $b = r_1^{t_1}	\cdots r_\ell^{t_\ell}$ with $r_1 < \cdots < r_\ell$ primes different from $p$ where $ord_{r_{i}^{t_i}}(x)=r_{i}^{h_i}$ with  $0\le h_i\le t_{i}-1$ for all $i$.
\end{enumerate}

Conversely, if \eqref{gkqb} holds with $b$ as in one of the items $(a)$--$(e)$ then the condition for $x$ stated in the 
corresponding item holds. 
\end{thm}

\begin{proof}
Clearly ($a$), ($b$), ($c$) and ($d$) are direct consequences of Theorem~\ref{WN g=b} and the divisibility properties of $\Psi_{b}(x)$ in the squarefree case given in Lemma \ref{x1 mod r}. On the other hand, ($e$) follows from Theorem~\ref{WN g=b} and Lemma~\ref{Lem r pow}. 
The remaining assertion is straightforward from Theorem~\ref{WN g=b} and 
Lemma \ref{gen case}.
The converse statements hold because Lemmas \ref{x1 mod r} and \ref{Lem r pow} 
are equivalences. 
\end{proof}

\subsubsection*{Prime values of Waring numbers}
If $p$ and $r$ are distinct primes, by ($a$) of Theorem \ref{exact val Hamm} we have
	\begin{equation} \label{gpr(r-1)a}
	g(\tfrac 1r \Psi_r(p^{ad}), p^{adr}) = r \quad \text{for any $a\in \N$} \qquad \Leftrightarrow \qquad
	d=ord_r(p).
	\end{equation}
By studying the congruence classes of $p$ modulo the first primes $r= 2, 3, 5, 7$ we have the following series of results. 
\begin{coro} \label{coro g=2}
If $p$ is an odd prime and $a$ is a positive integer then   
\begin{equation} \label{r, g=2}
g(\tfrac{p^a+1}2, p^{2a}) = 2.
\end{equation}
\end{coro}

\begin{proof}
Straightforward from ($a$) of Theorem \ref{exact val Hamm} with $r=2$.
\end{proof}

\begin{coro} \label{coro g=3}
If $p$ is a prime and $a$ is a positive integer we have  
\begin{equation} \label{r, g=3}
\begin{aligned}
& g(\tfrac{p^{2a}+p^a+1}3, p^{3a})=3, 		  & & \qquad \text{if \: $p \equiv 1 \!\! \pmod 3$}, \\[1mm]
& g(\tfrac{p^{4a}+p^{2a}+1}3, p^{6a})=3,      & & \qquad \text{if \: $p \equiv 2\!\! \pmod 3$}. 
\end{aligned}
\end{equation}
\end{coro}

\begin{proof}
Follows directly from ($a$) of Theorem \ref{exact val Hamm} with $r=3$, by noting that if $p\equiv 2 \pmod 3$, then $p^a\equiv (-1)^a \equiv 1 \pmod 3$ if and only if $a$ is even.  
\end{proof}

\begin{coro} \label{coro g=5}
If $p$ is a prime and $a$ is a positive integer then   
\begin{equation} \label{r, g=5}
\begin{aligned}
& g(\tfrac{p^{4a}+p^{3a}+p^{2a}+p^a+1}{5}, p^{5a})=5, 			& & \qquad \text{if \: $p \equiv  1 \!\!\pmod 5$}, \\[1mm]
& g(\tfrac{p^{8a}+p^{6a}+p^{4a}+p^{2a}+1}{5}, p^{10a})=5, 		& & \qquad \text{if \: $p \equiv  4 \!\!\pmod 5$}, \\[1mm]
& g(\tfrac{p^{16a}+p^{12a}+p^{8a}+p^{4a}+1}{5}, p^{20a})=5, 	& & \qquad \text{if \: $p \equiv 2,3 \!\!\pmod  5$}.
\end{aligned}
\end{equation}
\end{coro}

\begin{proof}
Follows from ($a$) of Theorem \ref{exact val Hamm} with $r=5$. If $p\equiv 4 \equiv -1 \pmod 5$, then $p^a \equiv (-1)^a \equiv 1 \pmod 5$ if and only if $a$ is even. If $p\equiv 2\pmod 5$ then $p^a \equiv 2^a\equiv 1 \pmod 5$ if and only if $a$ is a multiple of $ord_5(2)=4$. 
Thus, if $p\equiv 3\equiv -2 \pmod 5$, $p^a \equiv (-2)^a \equiv 1 \pmod 5$ if and only if $a$ is also a multiple of $4$, and the result follows.   
\end{proof}

Similarly as before we get
\begin{coro} \label{coro g=7}
	If $p$ is a prime and $a$ is a positive integer then   
	\begin{equation} \label{r, g=7}
	\begin{aligned}
& g(\tfrac{p^{6a}+p^{5a} + p^{4a} + p^{3a}+p^{2a}+p^a+1}{7}, p^{7a})=7, & & \qquad \text{if \: $p \equiv  1 \!\!\pmod 7$}, \\[1mm]
& g(\tfrac{p^{12a}+p^{10a}+p^{8a} + p^{6a}+p^{4a}+p^{2a}+1}{7}, p^{14a})=7, & & \qquad \text{if \: $p \equiv  6 \!\!\pmod 7$}, \\[1mm]
& g(\tfrac{p^{18a}+p^{15a}+p^{12a} + p^{9a}+p^{6a}+p^{3a}+1}{7}, p^{21a})=7, & & \qquad \text{if \: $p \equiv  2 \!\!\pmod 7$}, \\[1mm]
& g(\tfrac{p^{36a}+p^{30a}+p^{24a} + p^{18}+p^{12a}+p^{6a}+1}{7}, p^{42a})=7, & & \qquad \text{if \: $p \equiv  3,4,5 \!\!\pmod 7$}.
	\end{aligned}
	\end{equation}
\end{coro}

We now illustrate Corollaries \ref{coro g=2} -- \ref{coro g=7}.
\begin{exam} \label{exam g=2}
From \eqref{r, g=2}, for any $a\ge 1$ we have 
$$g(\tfrac{3^a+1}{2}, 3^{2a})= g(\tfrac{5^a+1}{2}, 5^{2a})= g(\tfrac{7^a+1}{2}, 7^{2a})=2.$$	
Thus, for $a=1,2,3$, we get $g(2,9)=g(5,81)=g(14,729)=2$ for $p=3$,
$g(3,25)=g(13,625)=g(63,15{.}625)=2$ for $p=5$ and $g(4,49)=g(25,2{.}401)=g(172,117{.}649)=2$ for $p=7$. 
\end{exam}

\begin{exam}
Since $11\equiv 2 \pmod 3$, $11\equiv 1 \pmod 5$, $11\equiv 4 \pmod 7$ and $13 \equiv 1 \pmod 3$, 
$13 \equiv 3 \pmod 5$, $13 \equiv 6 \pmod 7$, by Corollaries \ref{coro g=2}--\ref{coro g=7} we have 
\begin{align*}
& g(\tfrac{11+1}{2}, 11^2)=2, & & g(\tfrac{13+1}{2}, 13^2)=2, \\[1mm]  
& g(\tfrac{11^2+11+1}{3}, 11^6)=3, & & g(\tfrac{13^2+13+1}{3}, 13^3)=3, \\[1mm]  
& g(\tfrac{11^4+11^3+11^2+11+1}{5}, 11^5)=5, & & g(\tfrac{13^{16}+13^{12}+13^{8}+13^{4}+13^2+1}{5}, 13^{20})=5, \\[1mm]
& g(\tfrac{11^{36}+11^{30}+11^{24}+11^{18}+11^{12}+11^{6}+11}{7}, 11^{42})=7, & & 
g(\tfrac{13^{12}+13^{10}+13^{8}+13^{6}+13^{4}+13^{2}+1}{7}, 13^{14})=7.  
\end{align*}
\end{exam}

\subsubsection*{Squarefree values of Waring numbers}
Now, we deduce some results from parts ($b$) and ($c$) of Theorem \ref{exact val Hamm}.

\begin{coro} \label{coro 6}
If $p$ is an odd prime and $a$ is a positive integer then
\begin{equation} \label{coro 6}
g(\tfrac{p^{5a}+p^{4a}+p^{3a}+p^{2a}+p^a+1}6,p^{6a})=6, \qquad \text{if } p\ne 3,
\end{equation}
and 
\begin{equation} \label{coro 10}
\begin{aligned}
& g(\tfrac{p^{8a}+p^{7a}+p^{6a}+p^{5a}+p^{4a}+p^{3a}+p^{2a}+p^a+1}{10}, p^{10a})= 10, &&\qquad \text{if } p\equiv 1,4 \pmod{5},\\ 
& g(\tfrac{p^{16a}+p^{14a}+p^{12a}+p^{10a}+p^{8a}+p^{6a}+p^{4a}+p^{2a}+1}{10}, p^{20a})= 10, &&\qquad \text{if } p\equiv 2,3 \pmod{5}.  
\end{aligned}
\end{equation}
\end{coro}

\begin{proof}
	Let $b=2r$. If $r=3$, then $m=6a$. If $x=p^a$, then $x \equiv\pm 1 \pmod 3$ since $p\ne 3$ is prime. The result follows from ($b$) of the last theorem. Similarly for $r=5$ if $p \ne 5$.
\end{proof}

\begin{coro} \label{coro 15}
If $p\ne 3,5$ is a prime and $a$ a positive integer then we have 
\begin{equation} \label{coro 15}
\begin{aligned}
& g(\tfrac{1}{15} \Psi_{15}(p^{a}), p^{15a})=15, && \quad \text{if $p\equiv 1\pmod{15}$}, \\ 
& g(\tfrac{1}{15} \Psi_{15}(p^{2a}),p^{30a})=15, && \quad \text{if $p\equiv -1, \pm 4 \pmod{15}$}, \\  
& g(\tfrac{1}{15} \Psi_{15}(p^{4a}),p^{60a})=15, && \quad \text{if $p\equiv \pm 2, \pm 7 \pmod{15}$}.
\end{aligned}
\end{equation}
\end{coro}

\begin{proof}
We will use part~($c$) of Theorem~\ref{exact val Hamm}. 
Let $r=3$ and $r'=5$. Hence, $3\nmid 5-1$, $b=15$ and $m=15a$.
Since $(p^a,15)=1$ we only have to look at the cases $p \equiv d \pmod{15}$ for $d=1,2,4,7,8,11,13,14$.
If $p\equiv -1,\pm 4 \pmod{15}$ then $p^a \equiv (-1)^a \equiv (\pm 4)^a \equiv 1 \pmod{15}$ if and only if $a$ is even.   
If $p\equiv \pm 2,\pm 7 \pmod{15}$ then $p^a \equiv (\pm 2)^a \equiv (\pm 7)^a \equiv 1 \pmod{15}$ if and only if $a$ is a 
multiple of $4$.  
This completes the proof.
\end{proof}

Now we take the product of two primes $r, r'$ with $r \nmid r'-1$.
\begin{coro} \label{coro 21}
	If $p\ne 3,7$ is a prime and $a$ is a positive integer we have:   
	\begin{enumerate}[$(a)$]
	\item $g(\tfrac{1}{21} \Psi_{21}(p^{a}),p^{21a})=21$, 
	if $p\equiv 1\pmod{3}$ and $p\equiv 1,2,4\pmod{7}$. \sk 
		
	\item $g(\tfrac{1}{21} \Psi_{21}(p^{2a}), p^{42a})=21$, 
	if $p\equiv 1\pmod{3}$ and $p\equiv 3,5,6\pmod{7}$ or if $p\equiv 2\pmod{3}$. 
	\end{enumerate}
\end{coro}

\begin{proof}
	It is a direct consequence of ($d$) of the Theorem \ref{exact val Hamm}.
\end{proof}

\subsubsection*{Power of primes and general values of Waring numbers}
We now consider $b$ in general form.
Given a positive integer $b = r_1^{t_1} r_2^{t_2} \cdots r_\ell^{t_\ell}$ with $r_1, r_2, \ldots, r_\ell$ different primes, the radical of $b$ is $rad(b) = r_1 r_2 \cdots r_\ell$. 
We now exhibit an easy sufficient condition for \eqref{gkqb} to hold.
\begin{prop} \label{coro ptpm}
Let $p$ be a prime and $a,b \in \N$ such that $p\nmid b$. If $\varphi(rad(b)) \mid a$ 
then $g(\tfrac 1b \Psi_b(p^a),p^{ab})=b$. 
\end{prop}

\begin{proof}
Let $x=p^a$ and suppose $b=r_1^{t_1} \cdots r_\ell^{t_\ell}$ is the prime decomposition of $b$. By item ($f$) of Theorem \ref{exact val Hamm}, it is enough to show that $ord_{r_{i}^{t_i}}(x)=r_{i}^{h_i}$ for some $0\le h_i\le t_{i}-1$ for every $i$. 
Note that $\varphi(r_{i}^{t_i})=r_{i}^{t_i-1}(r_i-1)$ and 
$\varphi(rad(b)) = \varphi(r_1 \cdots r_\ell) = (r_1-1)\cdots (r_\ell-1)$. 
Since $r_i-1 \mid a$ for $i=1,\ldots,\ell$, because $r_i-1\mid\varphi(rad(b))$ and $\varphi(rad(b))\mid a$ by hypothesis, the Euler-Fermat's theorem implies that 
$$x^{r_{i}^{t_i-1}} = p^{ar_{i}^{t_i-1}} = (p^{\varphi(r_i^{t_i})})^{\frac{a}{r_i-1}}\equiv 1 \pmod {r_{i}^{t_i}} 
\qquad \text{for all } i=1,\ldots,\ell.$$
This implies that $ord_{r_i^{t_i}}(x)\mid r_i^{t_i}$ and then there is some $0\le h_i \le t_i -1$ such that $ord_{r_i^{t_i}}(x)=r_i^{h_i}$ for each $i=1,\ldots,\ell$.
\end{proof}

\begin{rem}
In particular, if $p$ is a prime not dividing an integer $b$, taking $a=\varphi(rad(b))$ in Proposition~\ref{coro ptpm} we have the expression 
\begin{equation} \label{g rad b} 
g(\tfrac 1b \Psi_{b}(p^{\varphi(rad(b))}),p^{b\varphi(rad(b))})=b
\end{equation}
depending only on $b$. Note that this complements and improves Proposition \ref{coro sobre}. In fact, \eqref{coro ptpm} gives explicit Waring pairs for any positive integer $b$. Moreover, given $b$, we find a smaller pair (hence an smaller field) than in \eqref{gkq=b} such that 
$g(k,q)=b$. 
\end{rem}

As a consequence, when $b$ is a prime power we get the following.
\begin{coro} \label{corolito}
 Let $p,r$ be different primes and let $a,t\in \mathbb{N}$. If $r-1\mid a$, then we have that
\begin{equation} \label{g=rt}	
g(\tfrac{1}{r^t}\Psi_{r^t}(p^{a}),p^{ar^t})=r^t.
\end{equation}
 In particular, for every $a \in \N$ and every odd prime $r$ we have % we obtain that
 \begin{equation} \label{g=2t}
	\begin{aligned}
		& g(\tfrac{1}{2^t}\Psi_{2^t}(p^{a}), p^{a2^t}) =2^t, %& \text{ for all } a\in \mathbb{N} 
		\\[1mm]
		& g(\tfrac{1}{r^t}\Psi_{r^t}(p^{r-1}), p^{\varphi(r^{t+1})}) =r^t. % &\text{ with $r$ odd}.
 \end{aligned}
 \end{equation}
\end{coro}

\begin{proof}
The first expression follows directly from Proposition \ref{coro ptpm} with $b=r^t$.
The remaining expressions follow directly from \eqref{g=rt} by taking $r=2$ and $a=r-1$ respectively.
\end{proof}

\begin{exam}
($i$) Taking $p=2$, $r=3$ and $t=2$ in \eqref{g=rt} we have
$$g(\tfrac{1}{9} \Psi_{9}(2^2), \, 2^{18} )=g(9{.}709, 2^{18}) =9.$$ 

\noindent ($ii$) 
Taking $r=2$ and $t=2,3,4$ in the first expression in \eqref{g=2t} we get
\begin{equation}
\begin{aligned}
g(\tfrac{p^{3a}+p^{2a}+p^a+1}4, p^{4a}) &= 4, \\[1mm] 
g(\tfrac{p^{7a}+\cdots +p^{2a} + p^a+1}8, p^{8a}) &= 8, \\[1mm] 
g(\tfrac{p^{15a}+\cdots+ p^{2a}+p^a+1}{16}, p^{16a}) &= 16,
\end{aligned}
\end{equation}
for any odd prime $p$ and any positive integer $a$ (for $t=1$ we get \eqref{r, g=2}).
\end{exam}

\begin{rem}
Note that Corollary \ref{corolito} gives infinite families of Waring numbers $(k_i, q_i)$ such that $g(k_i,q_i)=b$ with $b=r^t$. 
In particular, taking any sequence $\{a_i\}$ with $a_i \mid a_{i+1}$ for all $i$, we get a tower of finite fields 
$F_i = \ff_{q^{a_i}}$, $q=p^{r^t}$,
such that $g(k_i,q_i)=b$ for every $i$, where $k_i=\tfrac 1b \Psi_b(p^{a_i})$ and $q_i=q^{a_i}$, thus answering the question posed at the final of Section 4.
\end{rem}

\section{A lower bound for $g(k,p)$ from circulant GP-graphs}
As we have already mentioned in Section 2, there are three lower bounds for Waring numbers given by \eqref{lower 1} and \eqref{lower 2}. 
In this section, by using the known estimates of certain circulant graphs, we will find another lower bound for $g(k,p)$ in the case $p$ is prime.

\begin{prop} \label{lower bound}
Let $p$ be an odd prime and $h \in \N$. 
If $p\equiv 1 \pmod{2h}$ then 
\begin{equation} \label{lower 3}
g(\tfrac{p-1}{2h},p) \ge \tfrac 12 \sqrt[h]{h! \, p} - \tfrac{h+1}2.
\end{equation}
In particular, if $\sqrt[h]{h! \, p} > h+1$ we have $g(\tfrac{p-1}{2h},p) \ge \lceil \tfrac 12 \sqrt[h]{h! \, p} - \tfrac{h+1}2 \rceil$.
\end{prop}

\begin{proof}
A circulant graph is a graph whose adjacency matrix is circulant. Hence, circulant graphs can be seen as the Cayley graph of a cyclic group, i.e.\@ $\mathrm{Cay}(\Z_m, S)$ with $S$ any subset of $\Z_m$ not containing  $0$ (not necessarily symmetric). 
Notice that, since $\ff_p = \mathbb{Z}_p$, the graph 
$$\Gamma(k,p) = \mathrm{Cay}(\Z_p, R_k)$$ 
is circulant.
Recall that, when $p$ is odd, $R_k$ is symmetric if and only if $k \mid \frac{p-1}2$. 
By hypothesis, if we take $k=\frac{p-1}{2h} \in \mathbb{Z}$ then $k \mid \frac{p-1}2$, i.e\@ $R_k$ is symmetric in this case.

It is shown in \cite{WC} that the diameter of circulant graphs of a special form can be estimated from below. In fact, if 
$C_{m,\ell} = \mathrm{Cay}(\mathbb{Z}_{m},S_\ell)$ 
with $S_\ell = \{\pm 1,\pm a_2, \ldots,\pm a_\ell \}$
then 
\begin{equation}\label{diamcirc}
 \delta(C_{m, \ell}) \ge \tfrac 12 \sqrt[\ell]{\ell! \, m} - \tfrac{\ell+1}2.
\end{equation}

We now show that $\G(k,p)$ is the form $C_{p,\ell}$ for some $\ell$. 
Note that $S_\ell = T_\ell \cup (-T_\ell)$ where $T_\ell = \{1,a_2,\ldots,a_\ell\}$. 
Thus, we are lead to show that $R_k = \{x^k : x \in \Z_p^*\}$ is $S_\ell$ with $\ell = h = \frac{p-1}{2k}$. Clearly, 
$1\in R_k$ since $R_k$ is a multiplicative subgroup of $\ff_{p}^*$.
By symmetry, $x^k \in R_k$ if and only if $-x^k \in R_k$; and $x^k\ne -x^k$ since $p$ is odd. 
Also, it is clear that $|R_k|=\frac{p-1}k$ and hence $\ell = \frac{p-1}{2k}$.
Thus, $\G(k,p) = C_{p,\frac{p-1}{2k}}$. This and Theorem~\ref{gdiam} together imply 
$$g(k,p) = \delta(\Gamma_{k,p}) = \delta(C_{p,h}).$$
The result thus follows by \eqref{diamcirc} with $m=p$ and $\ell=h=\frac{p-1}{2k}$.
\end{proof}

As a direct consequence of the last proposition we get the following bounds for the smallest values of $h$ in \eqref{lower 3}, 
i.e.\@ $h=2,3,4,5$.

\begin{coro} \label{coro bound}
If $p$ is an odd prime we have
\begin{equation} 
\begin{aligned}
g(\tfrac {p-1}{4}, p)  & \ge \tfrac 12 \sqrt{2p} - \tfrac 32      & & \qquad \text{ if } p \equiv 1 \!\!\pmod{4}, \\[1mm]
g(\tfrac {p-1}{6}, p)  & \ge \tfrac 12 \sqrt[3]{6p} - 2           & &\qquad \text{ if } p \equiv 1 \!\!\pmod{6}, \\[1mm]
g(\tfrac {p-1}{8}, p)  & \ge \tfrac 12 \sqrt[4]{24p} - \tfrac 52  & &\qquad \text{ if } p \equiv 1 \!\!\pmod{8}, \\[1mm]
g(\tfrac {p-1}{10}, p) & \ge \tfrac 12 \sqrt[5]{120p} - 3         & &\qquad \text{ if } p \equiv 1 \!\!\pmod{10}.
\end{aligned}
\end{equation}
\end{coro}

\begin{exam}
	Consider the prime $37$, we have $37\equiv 1 \pmod{2h}$ for $h=2,3,6$. Hence, by Proposition \ref{lower bound} we have 
	$$g(9,37) \ge \lceil \tfrac{ \sqrt{2\cdot 37} -3}{2} \rceil, \qquad 
	g(6,37) \ge \lceil \tfrac{\sqrt[3]{6\cdot 37}}{2} -2 \rceil \qquad \text{and} \qquad 
	g(3,37) \ge \tfrac{\sqrt[6]{720 \cdot 37} -7}{2} .$$ 
	%%\end{align*}
	Thus, by the first item ($a$) in \S 2.2 we get 
	\begin{align*}
	3	&\le g(9,37) \le 9, \\[1mm]
	2	&\le g(6,37) \le 6, \\[1mm]
	0 &\le g(3,37) \le 3. 
	\end{align*}
	However $g(3,37))$ exists, i.e.\@ $g(3,37)\ge 1$ since $3$ is primitive divisor of $36$ (see Lemma \ref{equiv conn}).
\end{exam}

\subsubsection*{Comparison with other lower bounds}
We now compare our bound \eqref{lower 3} with the lower bounds given in \S 2.3. \sk 

\noindent ($a$)
It is difficult to compare the bounds in \eqref{lower 1} and \eqref{lower 3}. For $n=2h$, we have to compare the exponents of the roots $\varphi(2h)$ with $h$. Note that we always have $\varphi(2h) \le h$, with equality if and only if $h$ is a power of $2$. Even in this favorable case when $\varphi(2h)=h$, the main difficulty lies in the constant $c_n$. %where $k=\tfrac{p-1}{2h}$. 
Let $\omega$ be a primitive $n$-th root of unity in $\mathbb{Z}_p^*$ and $d=\varphi(n)$. For each 
$0\le j\le n-1$ on can write 
$\omega^j = \sum_{i=0}^{d-1} r_{i,j} \, \omega^i$. Then, the constant is defined by (see the proof of Claim 3.3 in \cite{GS}) 
$$c_n=\max_{i,j} \, |r_{i,j}|,$$ 
and we do not have an estimate for the growth of these numbers.

\sk
\noindent ($b$) The lower bounds in \eqref{lower 2} and \eqref{lower 3} are only comparable in the base cases, that is for $p$ odd and taking $r=2$ in \eqref{lower 2} and $h=1$  in \eqref{lower 3}. In both cases we get the same bound
$g(\tfrac{p-1}2, p) \ge \tfrac p2 -1$, although it is of course well-known that $g(\tfrac{p-1}2, p) = \tfrac{p-1}2$. \sk

\noindent ($c$)
The comparison between \eqref{lower 4} and \eqref{lower 3} is more interesting. The bound in \eqref{lower 4} seems to be better than the one in \eqref{lower 3} for integers $h$ which are not powers of $2$. %, although we have not checked this. 
However, for $h=2^t$ with $t>1$ and $p$ big enough, our bound does improve the one obtained by Cipra, Cochrane and Pinner.
We kindly thank Sam Chow for pointing this out to us. Indeed, for $h=2^t$ we have that $\varphi(2h)=h$ and $C_n=1$ in this case (see \eqref{Cn}), and hence we have  
$$ \sqrt[h]{h! p} - (h+1) >  (1-\tfrac{1}{p}) \sqrt[h]{p} \qquad \Longleftrightarrow \qquad 
\big(\sqrt[h]{h!} - (1- \tfrac 1p) \big) \sqrt[h]{p} > h+1.$$
Since $\sqrt[h]{h!}>1$ for $t>1$ we have $\sqrt[h]{h!} - (1-\tfrac 1p) >0$ and hence, for $p$ big enough, the expression on the right above holds.
Hence for $h=2^{t}$ we have 
$$g(\tfrac{p-1}{2^{t+1}} ,p) \ge \sqrt[2^t]{(2^t)!} \sqrt[2^t]{p} - \tfrac{h+1}2 \ge \tfrac{p-1}{p} \sqrt[2^t]{p}.$$ 
for any $t>1$. That is, the lower bound in \eqref{lower 3} is better that the one in \eqref{lower 4} for $n$ a power of 2.

\subsection*{Acknowledgments}
This work was initiated during a visit of the first author to UAM (Madrid) and UVa (Valladolid) in 2018. He wishes to thank the kind hospitality of both Professor Orlando Villamayor at UAM (and the grant MTM2015-68524-P) and Professor Antonio Campillo at UVa (and the grant of the research group GIR-Singacom).

\end{document}